\documentclass[10pt]{amsart}
\usepackage[english]{babel}
\usepackage{times,amsfonts,amsmath,amstext,amsbsy,amssymb,
  amsopn,amsthm,upref,eucal,amscd}
\usepackage[T1]{fontenc}
\usepackage[pdftex]{graphicx}

\newtheorem*{remark*}{Remark}
\newtheorem*{maintheorem}{Main Theorem}
\newtheorem*{theorem*}{Theorem}
\newtheorem{theorem}{Theorem}[section]
\newtheorem{lemma}[theorem]{Lemma}
\newtheorem{claim}[theorem]{Claim}
\newtheorem{corollary}[theorem]{Corollary}
\newtheorem{proposition}[theorem]{Proposition}

\numberwithin{equation}{section}

\theoremstyle{definition}
\newtheorem{definition}[theorem]{Definition}
\newtheorem{example}[theorem]{Example}

\newtheorem{remark}[theorem]{Remark}

\def\leq{\leqslant }
\def\geq{\geqslant}

\begin{document}

\title{Non-statistical rational maps}

\author{Amin Talebi}
\thanks{ The author is partially supported by ERC project 818737 \textit{Emergence of wild differentiable dynamical
systems}.
}
\address{ amin.s.talebi@gmail.com\\
Sharif University of Technology, Azadi St, Tehran, Iran and LAGA -CNRS, Universit\'e Paris 13, 93430, Villetaneuse, France.}

\date{\today}
\begin{abstract}
 We show that in the family of degree $d\geq 2$ rational maps of the Riemann sphere, the closure of strictly postcritically finite maps contains a (relatively) Baire generic subset of maps displaying maximal non-statistical behavior: for a map $f$ in this generic subset, the set of accumulation points of the sequence of empirical measures of almost every point in the phase space is the largest possible one that is, the set of all $f$-invariant measures. The proofs is based on a transversality argument which allows us to control the behavior of the orbits of critical points for maps close to strictly postcritically finite rational maps and also a new concept developed in the author's PhD thesis, that we call statistical bifurcation.   
\end{abstract}

\maketitle

\tableofcontents

\section*{Introduction}

Let $X$ be a compact metric space with a reference Borel probability measure $\mu $. For a point $x\in X$ and a map $f:X \rightarrow X$, the $n^{th}$ empirical measure 
\begin{equation*}
\mathbf e_{n}^{f}(x):=\frac{1}{n} \sum_{i=0}^{n-1} \delta_{f^i (x)},
\end{equation*}
describes the distribution of the orbit of $x$ up to the $n^{\text{th}}$ iteration in the phase space, which asymptotically may or may not converge. Let us call a map $f$ \emph{non-statistical} if there is a positive measure set of points with non-convergent empirical measures. There are several (but not too much) examples of differentiable dynamical systems showing this kind of behavior in the literature. One of the first examples of the non-statistical maps is the so-called Bowen eye \cite{takens1994heteroclinic}. It is the time one map of a vector field on $\mathbb{R}^2$ with an eye like open region such that Lebesgue almost every point in this region is non-statistical. There is another example by Colli and Vargas who have proved in \cite{CV01} that on any surface there is a $C^\infty$ diffeomorphism exhibiting a wandering domain and any point in this domain has non-convergent empirical measures. Their construction was based on perturbations of a diffeomorphism having a linear horseshoe with stable and unstable manifolds which are relatively thick and in tangential position. In this direction Kiriki and Soma in \cite{KS15} showed that the existence of wandering domain with non-statistical behavior happens densely in any Newhouse domain of $\text{Diff}^r(M)$ with $2\leq r <\infty$ on any surface $M$. Let us mention the work of Crovisier et al. \cite{crovisier2020empirical} which contains examples of non-statistical maps in the context of partially hyperbolic diffeomorphisms. There is an explicit example of a non-statistical diffeomorphism of the annulus introduced by Herman that can be found in \cite{Yoc-Her}. In this direction, in \cite{talebi20} the author has introduced a class of non-statistical dynamics in the context of the diffeomorphisms of the annulus and proved the Baire genericity of this maps within the space of Anosov-Katok diffeomorphisms. In this paper he has also developed an abstract setting to study the sufficient conditions for existence of non-statistical maps in a given family of dynamics.

There are also examples of non-statistical maps in the world of more specific families of dynamical systems (e.g. polynomial maps) where one looses the possibility of local perturbations as a possible mechanism to control the statistical behavior of the orbits. Hofbauer and Keller showed in \cite{hofbauer1990} that in the one parameter family of logistic maps $f_\lambda=\lambda x(1-x)$, there exists uncountably many parameters $\lambda\in [0,4]$ such that almost every $x\in [0,1]$ has non-convergent sequence of empirical measures. Indeed they showed in another paper \cite{HK2} that there are uncountably many maps in the logistic family with \textit{maximal oscillation } property: the empirical measures of almost every point in the phase space accumulates to each invariant probability measure of the dynamics. Another example of rigid dynamics displaying non-statistical behavior is the recent work of Berger and Biebler \cite{Berger-Biebler2020}. They prove the existence of real polynomial automorohisms of $\mathbb{C}^2  $ having some wandering Fatou component on which the dynamics has non-statistical behavior. Their work also contains a generalization of the result of Kiriki-Soma \cite{KS15} to the case of $r=\infty$ or $\omega$ and also the result in \cite{KNS2019}. \\
A natural direction to extend the result of Hofbauer and Keller is to go to the one dimensional complex dynamics and ask if there is any non-statistical and maximally oscillating rational map on the Riemann sphere. In this paper we show that the answer to this question is positive. We also prove that these maps are Baire generic inside a closed subset of rational maps which has positive measure, and so in particular there are uncountably many of these maps.

\subsection*{Statement of the result}
Denote by $Rat_d$ the space of rational maps of degree $d$ on the Riemann sphere $\hat{\mathbb{C}}$. A rational map is called \emph{postcritically finite} if all of its critical points have finite orbit. A map in $Rat_d$ is called \emph{strictly postcritically finite} if each of its critical points is eventually mapped to a repelling periodic orbit (this is equivalent to say that this postcritically finite map has no periodic critical point). The closure of strictly postcritically finite maps is a subset of bifurcation locus and is called \textit{maximal bifurcation locus}. Here is our main result in this paper:

\begin{maintheorem}\label{mainTheorem}
 For a Baire generic map $f$ in the maximal bifurcation locus, the set of accumulation points of the sequence of empirical measures is equal to the set of invariant measures of $f$ for Lebesgue almost every point.
\end{maintheorem}
 
 Although the set of the strictly postcritically finite rational functions is a countable union of 3-dimensional sub-varieties of $\mathrm{Rat}_d$, its closure -- the maximal bifurcation locus -- has recently been shown to have positive measure w.r.t. the volume measure on the the space of rational maps as a complex manifold (see \cite{AMGT17}).

Let us mention that the ideas we used in this paper can be applied to the case of one dimensional real dynamics as well, and provides us with another generalization of the result of Hofbauer and Keller in \cite{HK2}. In fact we can prove the Baire genericity of maximally oscillating maps in a compact subset of parameter space of logistic family,which is of positive Lebesgue measure.

Let us say a few words about the organization of the paper. In the first section the notion of statistical bifurcation is introduced. This notion has been developed by the author in his PhD thesis. The main theorem of this section is Theorem \ref{thm.generic BLambda}. At the end of this section in Theorem \ref{maximal-oscil}, we describe a special scenario that if it happens in a family of dynamical systems, then we can conclude the generic existence of \emph{maximally oscillating} dynamics within that family. In the rest of the paper we show that indeed this scenario happens for the rational maps in the maximal bifurcation locus. In Section \ref{sec.proof main theorem} we prove the main theorem of this paper using two propositions \ref{bif-to-per} and \ref{per-meas-density}. In Proposition \ref{bif-to-per} we show that any map in the maximal bifurcation locus statistically bifurcates toward the Dirac mass on an arbitrary periodic measure. This proposition is proved in Section \ref{sec.stat-bif-to-per-pt}. In the last section, Proposition \ref{per-meas-density} is proved in which we show the periodic measures are dense in the set of invariant measures for strictly postcritically finite maps.

\subsection*{Acknowledgment}
This paper is a part of author's PhD thesis which was written under a joint PhD program between Sharif university of technology and Universit\'e Paris 13. The author would very like to thank Pierre Berger for suggestion of this topic and his supervision while working on the problem, and also Meysam Nassiri for useful comments and discussions on this work, as well as his advisership during the PhD program.

\section{Formalization of the concept of statistical bifurcation}
In this section we are going to introduce an abstract setting for dealing with statistical bifurcation of a dynamical system. First let us talk about some basic notations and preliminaries. 

Let $X$ be a compact metric space endowed with a reference Borel probability measure $\mu $. For a compact metric space $(X,d)$, let us denote the space of probability measures on $X$ by $\mathcal M_1(X) $. This space can be endowed with weak-$^*$ topology which is metrizable, for instance with \emph{Wasserstien metric} $d_w$:
$$d_w(\nu_1,\nu_2):=\inf_{\zeta \in \pi(\nu_1,\nu_2)}\int_{X\times X} d(x,y) d\zeta \ ,$$
where $\nu_1$ and $\nu_2$ are two probability measures and $\pi(\nu_1,\nu_2)$ is the set of all probability measures on $X\times X$ which their projections on the first coordinate is equal to $\nu_1$ and on the second coordinate is equal to $\nu_2$. The  Wasserstein distance induce the weak-$^*$ topology on $\mathcal M_1(X)$ and hence the compactness of $(X,d)$ implies that $(\mathcal M_1(x),d_w)$ is a compact and complete metric space. We should note that our results and arguments in the rest of this note hold for any other metric inducing the weak-$^*$ topology on the space of probability measures. \\
For a point $x\in X$ and a map $f:X \rightarrow X$, the \emph{empirical measure} 
\begin{equation*}
 e_{n}^{f}(x):=\frac{1}{n} \sum_{i=0}^{n-1} \delta_{f^i (x)} 
\end{equation*}
describes the distribution of the orbit of $x$ up to the $n^\text{th}$ iteration in the phase space, which asymptotically may or may not converge.For a point $x\in X$ the set of accumulation points of the sequence of empirical measures $\{e^f_n(x)\}_{n}$ is always non-empty. If this sequence is convergent, we denote its limit by $e^f_{\infty}(x)$. In general this sequence may have a large set of accumulation points which we denote by $acc(\{e^f_n(x)\}_{n})$. We recall the following fact:
\begin{remark}
For any $x\in X$ we have $acc(\{e^f_n(x)\}_{n})\subset \mathcal M_1(f).$
\end{remark}

Now we come back to introduce our formalization for the concept of statistical bifurcation. Up to a fixed iteration, different points in the phase space have different empirical measures. We can investigate how the empirical measures $e^f_n(x)$ are distributed in $\mathcal M_1(X)$ with respect to the reference measure $\mu$ on $X$ and what is the asymptotic behavior of these distributions. To this aim consider the map $e^f_n:X\to \mathcal M _1(X)$ which sends each point $x\in X$ to its $n^{th}$ empirical measure, and push forward the measure $\mu$ to the set of probability measures on $X$ using this map: 
$$\hat{e}_n(f):=(e_n^f)_*(\mu).$$ 
The measure $\hat{e}_n(f)$ is a probability measure on the space of probability measures on $X$. We denote the space of probability measures on the space of probability measures by $\mathcal M_1(\mathcal M _1(X))$. We endow this space with the Wasserestein metric. Note that the compactness of $X$ implies the compactness of $\mathcal M_1(X)$ and hence the compactness of $\mathcal M_1(\mathcal M_1(X))$. So the sequence $\{\hat{e}_n(f)\}_{n\in \mathbb{N}}$ lives in a compact space and have one or possibly more than one accumulation points.
\begin{example}
For any $\mu$ preserving map $f:X\to X$, the sequence $\{\hat{e}_n(f)\}_{n\in \mathbb{N}}$ converges to a measure $\hat{\mu}$ which is the ergodic decomposition of the measure $\mu$.
\end{example}
\begin{example}
If $\nu$ is a physical measure for the map $f:X\to X$ whose basin covers $\mu$-almost every point, the sequence $\{\hat{e}_n(f)\}_{n\in \mathbb{N}}$ converges to the Dirac mass concentrated on the point $\mu\in \mathcal M_1(X)$, which we denote by $\delta_\mu$.
\end{example}
\begin{definition}
 We say a map $f$ is non-statistical in law if the sequence $\{\hat{e}_n(f)\}_{n\in \mathbb{N}}$ does not converge.
\end{definition}
Now let $\Lambda$ be a Baire space of self-mappings of $X$ endowed with a topology finer than $C^0$-topology. For each $f\in\Lambda$ the accumulation points of the sequence $\{\hat{e}_n(f)\}_{n\in \mathbb{N}}$ form a compact subset of $\mathcal M_1(\mathcal M_1(X))$ which we denote it by $acc( \{\hat{e}_n(f)\}_{n\in \mathbb{N}})$. This set can vary dramatically by small perturbations of $f$ in $\Lambda$:
\begin{example}\label{identity on S1}
Let $\Lambda$ be the set of rigid rotations on $\mathbb S^1$ and consider the Lebesgue measure as a reference measure. For the identity map $id$ on $\mathbb S^1$, the sequence $\{\hat{e}_n(id)\}_{n\in \mathbb{N}}$ is a constant sequence. Indeed for every $n\in\mathbb N$ we have:
\begin{align}\label{def e^}
\hat{e}_n(id)=\int_{\mathbb S^1}\delta_{\delta_x}dLeb.
\end{align}
So $acc( \{\hat{e}_n(id)\}_{n\in \mathbb{N}})$ is equal to $\{\int_{\mathbb S^1}\delta_{\delta_x}dLeb\}$. But for any irrational rotation $R_\theta$ (arbitrary close to the identity map), the sequence $\{\hat{e}_n(R_\theta)\}_{n\in \mathbb{N}}$ converges to $\delta_{Leb}$ which is a different accumulation point. 
\end{example}

In the previous example, for an irrational rotation $R_\theta$ close to the identity map , the empirical measures of almost every point start to go toward the Lebesgue measure, and hence the sequence $\{\hat{e}_n(R_\theta)\}_{n\in \mathbb{N}}$ goes toward $\delta_{Leb}$. To study the same phenomenon for the other dynamical systems, we propose the following definition. We recall that $\Lambda$ is a Baire space of self-mappings of $X$ endowed with a topology finer than $C^0$-topology and $\mu$ is a reference measure on $X$.
\begin{definition}
 For a map $f\in\Lambda$ and a probability measure $\hat{\nu}\in \mathcal M_1(\mathcal M_1(X))$, we say $f$ \emph{statistically bifurcates toward $\hat{\nu}$ through perturbations in $\Lambda$}, if there is a sequence of maps $\{f_k\}_k $ in $\Lambda$ converging to $f$ and a sequence of natural numbers $\{n_k\}_k$ converging to infinity such that the sequence $\{\hat{e}_{n_k}(f_k)\}_k$ converges to $\hat{\nu}\in\mathcal M_1(\mathcal M_1(X))$. 
\end{definition}

For the sake of simplicity, when the space $\Lambda$ in which we are allowed to perturb our dynamics is fixed, we say $f$ statistically bifurcates toward $\hat{\nu}$.

For any map $f\in\Lambda$, by $\mathcal{B}_{\Lambda,f}$ we denote the set of those measures $\hat{\nu}\in\mathcal M_1(\mathcal M_1(X))$ that $f$ statistically bifurcates toward $\hat{\nu}$ through perturbations in $\Lambda$.
\begin{definition}
 We say a map $f$ is statistically stable in law if the set $\mathcal{B}_{\Lambda,f}$ has only one element. Otherwise we say $f$ is statistically unstable in law.
\end{definition}
\begin{remark}\label{Rem.acc.subset}
By definition, it holds true that $acc(\{\hat{e}_n^f\}_n)\subset \mathcal{B}_{\Lambda,f}.$
\end{remark}

Let us remind some definitions that we need in the rest of this section.
Let $X$ and $Y$ be two topological spaces with $Y$ compact. Denote the set of all compact subsets of $Y$ by $\mathcal K (Y)$. A map $\phi:X\to \mathcal K (Y)$ is called lower semi-continuous if for any $x\in X$ and any $V$ open subset of $Y$ with $\phi(x)\cap V\neq \emptyset$, there is a neighbourhood $U$ of $x$ such that for any $y\in U$ the intersection $\phi(y)\cap V$ is non-empty. The map $\phi$ is called upper semi-continuous if for any $x\in X$ and any $V$ open subset of $Y$ with $\phi(x)\subset V$, there is a neighbourhood $U$ of $x$ such that for any $y\in U$ the set $\phi(y)$ is contained in $V$. And finally $\phi$ is called continuous at $x$ if it is both upper and lower semi-continuous at $x$.
To say $x$ is a continuity point of a set valued map $\phi:X\to \mathcal K (Y)$ with the above definition, is indeed equal to say $x$ is a continuity point of $\phi$ with considering $\mathcal K (Y)$ as a topological space endowed with Hausdorff topology. We recall the following theorem of Fort \cite{fort1951points} which generalizes a theorem related to real valued semi-continuous maps to the case of set valued semi-continuous maps:
\begin{theorem*}[Fort]
For any Baire topological space $X$ and compact topological space $Y$, the set of continuity points of a semi-continuous map from $X$ to $\mathcal K (Y)$ is a Baire generic subset of $X$.
\end{theorem*}

Now we study the properties of the map sending $f$ to the set $\mathcal B_{\Lambda,f}$:
\begin{lemma}\label{compact BLambda}
The set $\mathcal{B}_{\Lambda,f}$ is a compact subset of $\mathcal M_1(\mathcal M_1(X))$. 
\end{lemma}
\begin{proof}
By definition it is clear that the set $\mathcal{B}_{\Lambda,f}$ is closed. The compactness is a consequence of compactness of $\mathcal M_1(\mathcal M_1(X))$.
\end{proof}
Fixing a set of dynamics $\Lambda$, recall that by Lemma \ref{compact BLambda} the set $\mathcal B_{\Lambda,f}$ is a compact set. We can ask about dependence of the set $\mathcal B_{\Lambda,f}$ on the map $f$. The following lemma shows that this dependency is semi-continuous:
\begin{lemma}\label{lemma.BLambda.usc}
The map sending $f\in\Lambda $ to the set $\mathcal B_{\Lambda, f}$ is upper semi-continuous.
\end{lemma}
\begin{proof}
Let $\{f_n\}_{n }$ be a sequence converging to $f\in \Lambda$. We need to prove that if for each $n\in \mathbb N$, the map $f_n$ statistically bifurcates toward a measure $\hat{\nu}_n\in \mathcal M_1(\mathcal M_1(X))$ through perturbations in $\Lambda$ and the sequence $\{\hat{\nu}_n\}_n$ is convergent to a measure $\hat{\nu}$, then the map $f$ statistically bifurcates toward $\hat{\nu}$ through perturbations in $\Lambda$. 
So the proof is finished by observing that for $n$ large enough, small perturbations of the map $f_n$ are small perturbations of the map $f$, and $\hat{\nu}_n$ is close to $\hat{\nu}$.
\end{proof}

To each map $f\in\Lambda$, one can associate the set of accumulation points of the sequence $\{\hat{e}_n(f)\}_{n\in \mathbb{N}}$ which is a compact subset of $\mathcal M_1 (\mathcal M _1(X))$. By looking more carefully at the Example \ref{identity on S1}, we see that this map is neither upper semi-continuous nor lower semi-continuous. However if we add all of the points of this sequence except finite ones, to its accumulation points, we obtain a semi-continuous map:
\begin{lemma}\label{E lsc}
The map $\mathcal E_k: \Lambda \to \mathcal K ({\mathcal M_1(\mathcal M_1(X))})$ defined as 
$$\mathcal E_k(f):=\overline{\{\hat{e}_n(f)|n\in \mathbb{N},n>k\}},$$
is lower semi-continuous.
\end{lemma}
\begin{proof}
Let $V$ be an open subset of $\mathcal M_1(\mathcal M_1(X))$ intersecting $\mathcal E_k (f)$. So there is $n\in \mathbb N$ such that $\hat{e}_n(f)\in V$. But note that the map $f\mapsto\ \hat{e}_n(f)$ is continuous and so there is a neighborhood $U$ of $f$ so that for any $g\in U$, we have $\hat{e}_n(g)\in V$ and so $\mathcal E_k (g)$ intersects the set $V$. This shows that $\mathcal E_k$ is lower semi-continuous.
\end{proof}

The following lemma is an interesting consequence of lemma \ref{E lsc} which shows how the set $\mathcal{E}_k(f)$ depends on the dynamics $f$. 

\begin{lemma}\label{cor of Fort}
For any $k\in\mathbb{N}$ the set of continuity points of the map $\mathcal E_k$ is a Baire generic subset of $\Lambda$ .
\end{lemma}

This lemma gives a view to the statistical behaviors of generic maps in any Baire space of dynamics: for a generic map, the statistical behavior that can be observed for times close to infinity can not be changed dramatically by small perturbations.  
\begin{proof}
Using Lemma \ref{E lsc}, this is a direct consequence of Fort's theorem. 
\end{proof}
\begin{lemma}\label{lem.B.subset}
For a Baire generic map $f\in \Lambda$ it holds true that 
$\mathcal B_{\Lambda,f}\subset \mathcal E_k (f).$
\end{lemma}

\begin{proof}
Observe that by Lemma \ref{cor of Fort}, for a generic map $f\in\Lambda$ we have 
$$\limsup_{g\to f} \mathcal{E}_k(g)=\mathcal{E}_k(f).$$
On the other hand according to the definition of $\mathcal B_{\Lambda,f}$ for any $k\in\mathbb{N}$  
$$\mathcal B_{\Lambda,f} \subset \limsup_{g\to f} \mathcal{E}_k(g),$$
and this finishes the proof.
\end{proof}

The following theorem reveals how two notions of statistical instability in law and being non-statistical in law are connected. 

\begin{theorem}\label{thm.generic BLambda}
Baire generically, $\mathcal B_{\Lambda,f}$ is equal to $acc(\{\hat{e}_n^f\}_n)$.
\end{theorem}

 \begin{proof}
First note that $acc(\{\hat{e}_n^f\}_n)=\bigcap_{k\in\mathbb{N}} \mathcal E_k$. By lemma \ref{lem.B.subset} the set $\mathcal B_{\Lambda,f}$ is included in $\mathcal {E}_k$ for generic $f\in \Lambda$ and hence as the intersection of countably many generic set is generic, for a generic map $f$ it holds true that 
$$\mathcal B_{\Lambda,f}\subset \bigcap_{k\in\mathbb{N}} \mathcal E_k=acc(\{\hat{e}_n^f\}_n).$$

The other side of the inclusion $acc(\{\hat{e}_n^f\}_n)\subset \mathcal B_{\Lambda,f}$ follows from the definition. 
 \end{proof}

This short and simple proof was suggested by Pierre Berger. There is also a different proof of this theorem in the PhD thesis of the author \cite{talebi20}.

Now note that if we have any information about the set $\mathcal B_{\Lambda, f}$ then by using theorem \ref{thm.generic BLambda} we may translate it to some information about $acc(\{e^f_n(x)\}_n)$ for generic $f\in \Lambda$. In particular we obtain the following corollary:
\begin{corollary}\label{cor.gen-instability=non-stat}
The set $\Lambda$ contains a Baire generic subset of maps that are statistically unstable in law iff it contains a Baire generic subset of maps which are non-statistical in law. 
 \end{corollary}
In the following, we explain a scenario through which this lemma can be used to deduce information about the behavior of generic maps. This scenario is indeed what happens in the example of non-statistical rational maps we introduce in this paper.\\

Suppose the initial map $f\in \Lambda$ has an invariant measure $\nu$ such that by a small perturbation of the map, the empirical measures of arbitrary large subset of points is close to $\nu$ for an iteration close to infinity. For instance you can think of the identity map on $\mathbb{S}^1$ which can be perturbed to an irrational rotation for which the empirical measures of almost every point converges to the Lebesgue measure or it can be perturbed to a Morse-Smale map having one attracting fixed point and so the empirical measures of almost every point converges to the Dirac mass on that attracting fixed point. 
 These measures could be interpreted as potential physical measures with full basin for our initial dynamics. We denote this set of measures by $\mathcal M_{\Lambda,f}$ which are defined more precisely as bellow:
$$\mathcal M_{\Lambda, f}:=\{\nu\in \mathcal M_1(X)|\delta_{\nu}\in\mathcal B_{\Lambda,f}\}.$$
\begin{theorem}\label{generic stat behav}
Let $\Lambda$ be a Baire space of self-mappings of $X$ endowed with a topology finer than $C^0$-topology. For a Baire generic map $f\in\Lambda$ the empirical measures of $\mu$ almost every point $x\in X$, accumulates to each measure in $\mathcal M _{\Lambda,f}$ or in other words:
\begin{align}\label{genericity of large accumulation}
for\ \mu-a.e. \ x\in X,\quad \mathcal M_{\Lambda,f}\subset acc(\{e^f_n(x)\}_{n }).
\end{align}
\end{theorem}
\begin{proof}
To prove the theorem it suffices to show that if $f\in \Lambda$ is a continuity point of the map $\mathcal E_1$ it satisfies condition \eqref{genericity of large accumulation}. Indeed, by Corollary \ref{cor of Fort} the continuity points of the map $\mathcal E_1$ form a Baire generic subset of $\Lambda$.  

Take any measure $\nu$ inside $\mathcal M_{\Lambda,f}$. Theorem
\ref{thm.generic BLambda} implies that $\delta_\nu\in \mathcal E_1(f)$. Now there are two possibilities, either there is a number $n\in \mathbb N$ such that $\hat{e}_n(f)=\delta_\nu$ or not. If not, there is a sequence $\{n_i\}_{i }$ converging to infinity such that 
$$\lim_{i\to\infty}\hat{e}_{n_i}(f)=\delta_\nu.$$ In this case for a small neighbourhood $U\subset\mathcal M_1(X)$ of $\nu$, we have:
$$\lim_{i\to \infty} \hat{e}_{n_i}(f)(U)=\delta_\nu(U)=1.$$
 Noting that we can write $\hat{e}_n(f)$ as bellow
 \begin{equation}\label{equ def e hat}
     \hat{e}_{n_i}(f)=(\int_{X}\delta_{e_{n_i}^f(x)}d\mu),
 \end{equation}
 we obtain
$$\lim_{i\to \infty}(\int_{X}\delta_{e_{n_i}^f(x)}d\mu)(U)=1.$$
So for $\mu$-almost every point $x\in X$ we have:
$$ \lim_{i\to\infty}\delta_{e_{n_i}^f(x)}(U)=1. $$
Since $U$ is an arbitrary neighbourhood, we can conclude that for $\mu$-almost every point $x\in X$ , the measure $\delta_\nu$ is contained in the accumulation points of the sequence $\{\delta_{e_{n_i}^f(x)}\}_i$. But this is equal to say that $\nu $ is in the accumulation points of the sequence $\{e_{n_i}^f(x)\}_i$. So the measure $\nu$ is an accumulation point of the sequence $\{e_n^f(x)\}_{n }$, which is what we sought.

It remains to check the case that there is a number $n\in\mathbb N$ such that $\hat{e}_n(f)=\delta_\nu$ . In this case, using equation (\ref{equ def e hat}) we obtain:
$$\int_{X}\delta_{e_n^f(x)}d\mu=\delta_\nu,$$
so $\mu$-almost every $x\in X$ has the property that the measure $e_n^f(x)$ is equal to $\nu$. Recalling that $\nu$ is an $f$-invariant measure, every point $x$ with this property should be a periodic point and $\nu$ should be the invariant probability measure supported on the orbit of $x$. So obviously the measure $\nu$ lies in the accumulation points of the sequence $\{e_n^f(x)\}_{n }$ . This finishes the proof.
\end{proof}

Using Theorem \ref{generic stat behav} we are able to translate any information about the set $\mathcal M_{\Lambda,f}$ for $f$ in a generic subset of maps in $\Lambda$ to information about the statistical behavior of $\mu$-almost every point for a generic subset of maps . 

The following lemma shows how the set $\mathcal M_{\Lambda,f}$ depends on the map $f$:
\begin{lemma}\label{MLambda,f usc}
The map sending $f\in\Lambda $ to the set $\mathcal M_{\Lambda, f}$ is upper semi-continuous.
\end{lemma}
\begin{proof}
Let $\{f_n\}_{n }$ be a sequence converging to $f\in \Lambda$. We need to prove that if for each $n\in \mathbb N$, the map $f_n$ statistically bifurcates toward a measure $\delta_{\nu_n} \in \mathcal M_{\Lambda, f_n}$ through perturbations in $\Lambda$ and the sequence $\{\nu_n\}_n$ is convergent to a measure $\nu$, then the map $f$ statistically bifurcates toward $\delta_\nu$ through perturbations in $\Lambda$. Considering the fact that for $n$ large enough, small perturbations of the map $f_n$ are small perturbations of the map $f$, the rest of the proof is straight forward.
\end{proof}
Now let us see what is the consequence of this lemma and Theorem \ref{generic stat behav} together with the assumption of density of maps in $\Lambda$ for which the dynamics statistically bifurcates toward the Dirac mass on any invariant measure. Before that, we introduce the following definition which was used for the first time by Hofbauer and Keller in \cite{HK2}:
\begin{definition}
A map $f\in\Lambda$ is said to have \emph{maximal oscillation} if the empirical measures of almost every point accumulates to all of the invariant measures in $\mathcal{M}_1(f)$.
\end{definition}

\begin{theorem}[Maximal oscillation]\label{maximal-oscil}
Assume that there exist a
dense set $D$ in $\Lambda$ such that for every $f\in D$ it holds true that $M_{\Lambda,f} = M_1(f)$.
Then a Baire generic $g\in \Lambda$ has maximal
oscillation. 
\end{theorem}
\begin{proof}
By Proposition \ref{MLambda,f usc} the map sending $f$ to $\mathcal M_{\Lambda,f}$ is semi-continuous. The map sending $f$ to $\mathcal M_1(f)$ is also upper semi-continuous. So by applying Fort's theorem we can find a Baire generic subset $\mathcal B\subset \Lambda$ such that any $f$ in this set is a continuity point for each of these maps. Now we can approach each map $f$ in $\mathcal B$ by maps $g$ in $D$, for which we know $\mathcal M_1(g)$ and $\mathcal M_{\Lambda,g}$ co-inside. So $\mathcal M_1(f)$ and $\mathcal M_{\Lambda,f}$ co-inside as well. By Theorem \ref{generic stat behav} we know there is a Baire generic subset of $\Lambda$ that for any map $f$ in this set the empirical measures of $\mu$ almost every point $x\in X$ accumulates to each of measures in $\mathcal M _{\Lambda,f}$ . The intersection of this Baire generic set with $\mathcal B$ is still a Baire generic set and for a map $f$ in this intersection the empirical measures of $\mu$ almost every point $x\in X$ accumulates to each of measures in $\mathcal M_1(f)$.
\end{proof}

\section{Proof of Main Theorem}\label{sec.proof main theorem}
First let us introduce the following definitions and notations that we use while dealing with the dynamics of rational maps.We say a point $x\in X$ is \emph{preperiodic} if it is mapped to a periodic point $p$ after some iterations. In this case we may say the point $x$ is \emph{preperiodic to} the periodic point $p$. 
We say a measure $\mu \in \mathcal M_1(f)$ is an \textit{invariant periodic measure} if it is supported on the orbit of a periodic point.

The space of degree $d$ rational maps $\text{Rat}_d$ is a $2d+1$ dimensional complex manifold. To see this, note that we can parametrize it around any element $\frac{P}{Q}\in \text{Rat}_d$ using the coefficients of the polynomials $P$ and $Q$. These two polynomials have terms up to degree $d$ so there is $2d+2$ coefficients. But note that multiplying both $P$ and $Q$ by a constant does not change the rational map, so the dimension is $2d+1$. 
\begin{remark}
Any degree $d$ rational map has $2d-2$ critical points counting with multiplicity.
\end{remark}
 Here are some notations:
\begin{itemize}
\item $\mathrm{Per}(f)$ for the set of the periodic points of a map $f$.
\item$\mathcal{C}(f)$ for the set of critical points of a rational map $f$.
\item$\mathcal{P}(f)$ for the postcritical set of a rational map which is defined as follows:
$$\mathcal{P}(f):=\bigcup_{n\geq1}f^n(\mathcal{C}).$$
\item$\kappa_d$ for the set of those maps in $\mathrm{Rat}_d$ which has no periodic critical points (or no attracting periodic point). 
\item$\kappa_d^*$ for the set of those maps in $\kappa_d$ for which all the critical points are simple and the postcritical set does not contain any critical point.
\end{itemize}

To prove the main theorem we show that the maps in $\kappa_d$ enjoy from two nice properties stated in the following propositions. The proofs of these propositions is postponed to the next sections. \\
The first proposition is related to the statistical behavior of perturbations of the maps in $\kappa_d$ within the maximal bifurcation locus $\overline{\kappa_d}$.
\begin{proposition}\label{bif-to-per}
Assume $f$ is a map in $\kappa_d$. Then for any periodic point $p\in\mathrm{Per}(f)$, $f$ statistically bifurcates toward $\delta_{e_{\infty}^f(p)}$ through perturbations in $\overline{\kappa_d}$.
\end{proposition}
Note that a rational map of degree greater than one, has always (infinitely) many different periodic orbits, and in fact, the set of periodic points is dense in the Julia set. So the set of periodic measures contains infinitely many elements, each one corresponds to a periodic cycle. The next proposition states that for a map in $\kappa_d$, the set of periodic measures is in some sense maximal. 
\begin{proposition}\label{per-meas-density}
For any strictly postcritically finite rational map $f$, the set of invariant probability measures which are supported on the orbit of a periodic point, is dense in the set of all invariant measures $\mathcal{M}_1(f)$.
\end{proposition}
\begin{remark}
In the proof of Proposition \ref{per-meas-density} we will see that every periodic point for a strictly postcritically finite map is repelling. 
\end{remark}
From these two propositions we conclude that for a map in $\kappa_d$ the set of measures to which $f$ statistically bifurcates is maximal. 
\begin{corollary}\label{bif-to-M1}
Any map in $\kappa_d$ statistically bifurcates toward the Dirac mass on any of its invariant measures through perturbations in $\overline{\kappa_d}$ or in other word:
$$\forall f \in \kappa_d, \quad \{\nu\in \mathcal M_1(X)|\delta_{\nu} \in \mathcal B_{\Lambda,f}\}=\mathcal{M}_1(f).$$
\end{corollary}

\begin{remark}
We use the word maximal because a map $f$ cannot statistically bifurcates toward the Dirac mass on a measure that is not $f$-invariant. 
\end{remark}
Let us show how this corollary together with Proposition \ref{maximal-oscil} implies the main theorem.
\begin{proof}[End of proof of Main Theorem]
By Corollary \ref{bif-to-M1}, every map in $\kappa_d$ bifurcates toward the Dirac mass on each of its invariant measures. So by Proposition \ref{maximal-oscil}, for a generic $f$ in $\overline{\kappa_d}$, 
the set of accumulation points of the sequence of empirical measures of Leb-almost every point, is equal to the whole set of invariant measures. This finishes the proof.
\end{proof}


\section{Statistical bifurcation toward periodic measures}\label{sec.stat-bif-to-per-pt}
The aim of this section is to prove Proposition \ref{bif-to-per}.
First let us recall the following two theorems from \cite{BE09} and \cite{buff2013perturbations}.  

We recall that a \emph{Latt\`es map} $f$ is a postcritically finite map which is semi-conjugated to an affine map $A: z\mapsto az+ b$ on a complex torus $\mathcal T$, via a finite to one holomorphic semi conjugacy $\Theta$ :
\[\Theta \circ A= f\circ \Theta\; .\] 
A latt\`es map $f$ is \emph{flexible} if we can  choose $\Theta$ with degree $2$ and $A$ with $a>1$ integer.  

We denote by $\mathcal L_d$ the set of flexible Latt\`es maps of degree $d$. We refer the reader to the paper of Milnor \cite{milnor2006lattes} for further discussion on Latt\'es maps.\\
We observe that:
\[\mathcal L_d\subset \kappa^*_d\subset \kappa_d\, .\]
On the other hand:
\begin{theorem*}[Buff-Epstein]\label{theoBE} The following inclusion holds true:
\[\kappa_d \setminus \mathcal L_d \subset \overline{\kappa^*_d\setminus \mathcal L_d}\, .\]
\end{theorem*}
This theorem is a part of the main theorem of \cite{BE09}, whereas the following one is the main theorem of \cite{buff2013perturbations}.

\begin{theorem*}[Buff-Gauthier]\label{theoBG}
A flexible Latt\`es map can be approximated by strictly postcritically finite rational maps which are not a flexible Latt\`es map:
\[\mathcal L_d \subset \overline{\kappa_d\setminus \mathcal L_d}\, .\]
\end{theorem*}
These two theorems imply:
\begin{corollary}\label{cor-approx}
Any strictly postcritically finite rational map $f\in \kappa_d$ can be approximated by maps in $\kappa_d^*$ which are not flexible Latt\`es map:
\[\kappa_d  \subset \overline{\kappa^*_d\setminus \mathcal L_d}\, .\]
\end{corollary}
\begin{proof}
By Proposition \ref{theoBE}, if $f$ is not a flexible Latt\`es map we are done. If $f$ is a flexible Latt\`es map, then first by Proposition \ref{theoBG}, it can be approximated by a strictly postcritically finite map which is not a flexible Latt\`es map. Now using Proposition \ref{theoBE} again, we are done.
\end{proof}
Corollary \ref{cor-approx} enables us to transfer the following property of maps in $\kappa_d^*\setminus \mathcal L_d$ to those in $\kappa_d$, in order to deduce Proposition \ref{bif-to-per}.
\begin{lemma}[Main lemma]\label{Main lemma}
Let $f$ be a map in $\kappa_d^*\setminus \mathcal L_d$. Then for any periodic point $q\in\mathrm{Per}(f)$, $f$ statistically bifurcates toward $\delta_{e_{\infty}^f(q)}$ through perturbations in $\overline{\kappa_d}$.
\end{lemma}

\begin{figure}\label{Fig-Lambda1}
    \centering
    \includegraphics[scale=0.5]{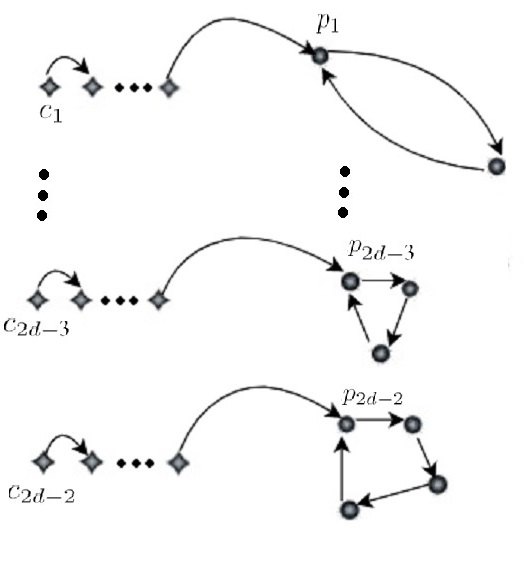}
    \caption{The dynamics of the initial map $f\in \kappa_d^*\setminus \mathcal L_d$. }
    \end{figure} 

We will prove this lemma below, before this let us prove Proposition \ref{bif-to-per}.
\begin{proof}[Proof of Proposition \ref{bif-to-per}]
For any map $f$ in $\kappa_d$, any periodic point $p$ is repelling, and its hyperbolic continuation is well defined and so the periodic measure supported on its cycle has a well defined continuation for $f'$ close to $f$.  

 Hence, to show that $f$ statistically bifurcates toward $\delta_{e_{\infty}^f(p)}$ through perturbations in $\overline{\kappa_d}$, it is enough to show that there is some map $f'$ in $\overline{\kappa_d}$ arbitrary close to $f$ that statistically bifurcates toward the Dirac mass on the continuation $e_{\infty}^{f'}(p)$ of this measure. But by Corollary \ref{cor-approx}, arbitrary close to $f$ we can find elements of $\kappa_d^*\setminus \mathcal L_d$, and by Main lemma, these maps statistically bifurcate toward the Dirac mass on any of their periodic measures, in particular, to the Dirac mass on the continuation $e_{\infty}^{f'}(p)$. This finishes the proof of Proposition \ref{bif-to-per}.
\end{proof}

\begin{proof}[Proof of Lemma \ref{Main lemma}] 
Denote by $c_1(f)$,...,$c_{2d-2}(f)$ the  $2d-2$ distinct critical points of $f$. There are repelling periodic points $p_1(f)$,...,$p_{2d-2}(f)$ and positive integers $n_1$,...,$n_{2d-2}$ such that as it is shown in Figure $1$
$$ f^{n_i}(c_i(f))=p_i(f),\ \ \ \ 1\leq i\leq 2d-2\ $$

The critical points are simple and periodic points are repelling so by the implicit function theorem, for any $1\leq i \leq 2d-2$ there are 
\begin{itemize}
 \item analytic germ $c_i:(\mathrm{Rat}_d,f) \to \hat{\mathbb{C}}$ following the critical point of $f$ as $g$ ranges in a neighbourhood of $f$ in $\mathrm{Rat}_d$ and
 \item analytic germ $p_i:(\mathrm{Rat}_d,f) \to \hat{\mathbb{C}}$ following the periodic point of $f$ as $g$ range in a neighbourhood of $f$ in $\mathrm{Rat}_d$.
\end{itemize}
Let $F:(\mathrm{Rat}_d,f) \to \mathbb{C}^{2d-2}$ and $G:(\mathrm{Rat}_d,f) \to \mathbb{C}^{2d-2}$ be defined by:
$$ F(g) = \begin{pmatrix}
	 F_{1}(g) \\
	 \vdots \\
	 F_{2d-2}(g)
	\end{pmatrix}
\  \mathrm{with}\ \    F_{j}(g):=f^{n_{j}}(c_{j}(g)) \  \ \mathrm{and}\ \   P(g) = 
	\begin{pmatrix}
	p_1(g) \\
	\vdots \\
	p_{2d-2}(g)
	\end{pmatrix}.
$$ 

Denote by $D_{f}F$ and $D_{f}P$ the differentials of $F$ and $P$ at $f$. The following transversality result has been proved many times, see for example \cite{vS00}. We recall a version which is presented in \cite{BE09}:
\begin{proposition}[Epstein]
The linear map
$$D_{f}F-D_{f}P: T_{f}\mathrm{Rat}_d \to T_{p_{1}(f )}\hat{\mathbb{C}}\times \ldots \times T_{p_{2d-2}(f )}\hat{\mathbb{C}}$$ 
has rank $2d-2$. The kernel of $D_{f }F-D_{f }P$ is the tangent space to the subset of $Rat_d$ which is formed by those maps that are conjugate to $f$ by a M\"obius transformation.
\end{proposition}
This nice property enables us to have control on the orbits of the critical points while perturbing the dynamics.

\begin{proposition}\label{1dim-famil}
For any map in $f\in\kappa_d^*$ which is not a flexible Latt\`es map, there is a holomorphic, one-dimensional family $\{f_\Lambda\}_{\Lambda \in \mathbb{D}}$ such that $f_0=f$, and except $c_1(f_0)$, the other critical points are persistently preperiodic through this family.
\end{proposition}
 \begin{proof}
For any $1\leqslant i \leqslant 2d-2$ let the map $\phi _{i}:U_{i}\to \mathbb{C}$ be a local coordinate around $p_{i}(f)$, such that $\phi_{i}(p_{i}(f ))=0.$ Then by the previous proposition the derivative of the map 
$$\Phi:= \begin{pmatrix} 
	 \phi_1 \circ F_1-\phi_1 \circ p_1 \\
	 \vdots \\
	 \phi_{2d-2} \circ F_{2d-2}-\phi_{2d-2} \circ p_{2d-2}
	 \end{pmatrix}.
$$
at $f $ has full rank, so if we denote the $\epsilon$-neighbourhood of zero in the complex plane by $\mathbb{C}_\epsilon$, then by Rank theorem there is a one dimensional holomorphic family $\{f_{\lambda}\}_{\lambda \in \mathbb{D}_{\epsilon}}$ for $\epsilon>0$ sufficiently small, such that
$ \Phi(f_{\lambda})=(\lambda,0,0,\ldots,0).$ So for any $\lambda \in \mathbb{D}_{\epsilon}$ and for any $j\neq 1$ we have
$f_{\lambda}^{n_{j}}(c_{j}(f_{\lambda}))=p_{j}(f_{\lambda}).$ And obviously this equality does not hold true for critical point $c_{1}(f_{\lambda})$. By reparametrizing the family, we obtain a family $\{f_{\lambda}\}_{\lambda \in \mathbb{D}}$ enjoying the desired properties.
\end{proof}

Let us consider a family $\{f_{\lambda}\}_{\lambda \in \mathbb{D}}$ coming from Proposition \ref{1dim-famil}, and denote the bifurcation locus of this family by $B$ recalling that:

\begin{definition}
 The bifurcation locus of a family consists of those parameters that the dynamics is not structurally stable within that family.
\end{definition}

\begin{remark}
The bifurcation locus $B$ is non-empty and in particular contains 0.
\end{remark}
\begin{proof}
The family we are considering is so that $c_1(f_{\lambda})$ is sent to $p_1(f_{\lambda})$ by $n_1$ iteration for $\lambda=0$ , but this does not happen for $\lambda\neq 0$. So $f_0$ is not structurally stable in this family. 
\end{proof}
\begin{remark}
 Since for every $\lambda$ sufficiently close to zero the orbit of each critical point other than $c_1(f_0)$ is finite, $c_1(f_{\lambda})$ is disjoint from the orbit of the other critical points. So by reparameterizing the maps associated to the parameters close to zero, we can assume that every map in the family satisfies this property. This is a technical assumption that we will use later.
\end{remark}

\begin{lemma}\label{argument_normality}
For every periodic point $q(f_0)$ of the map $f_0$, there is a parameter $\lambda^*$ in the bifurcation locus $B$ arbitrary close to zero, such that $c_1(f_{\lambda^*})$ is preperiodic to $q(f_{\lambda^*})$.

\end{lemma}
\begin{figure}\label{Fig-Lambdastar}
    \centering
    \includegraphics[scale=0.5]{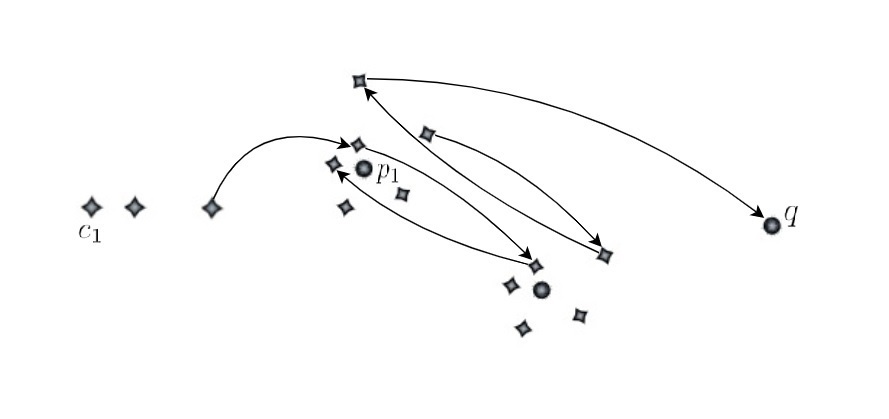}
    \caption{For the map $f_{\lambda^*}$ which is a suitable perturbation of $f_0=f$, the orbit of the critical point $c_1$ is repelled by the cycle of the periodic point $p_1$ and eventually land on $q$ (see lemma \ref{argument_normality}).}
    \label{Lambda-star}
\end{figure} 
\begin{proof}
The proof uses the well known normal family argument. Let $U$ be a small neighbourhood around $0\in \mathbb{D}$. Recalling that the parameter zero is in the bifurcation locus, by Theorem 4.2 of  McMullen's paper \cite{McMullen}, there is $j$ for which the family
$\{\lambda\in U\mapsto f_{\lambda}^n(c_{j}(f_{\lambda}))\}_{n\in \mathbb{N}}$ is not normal. But by Proposition \ref{1dim-famil}, this family is eventually periodic for $j\neq 1$ and hence it is normal. So for $j=1$, it is not normal. Using this we are going to find $\lambda^*$ in $U$ such that $c_{1}(f_{\lambda^*})$ is preperiodic to $q(f_{\lambda^*})$. If this holds for $\lambda^*=0$ we are done. If not:
\begin{claim}
If $c_{1}(f_{0})$ is not preperiodic to $q(f_0)$, then any pre-image of $q(f_0)$ depends holomorphiclly on the parameter in a neighbourhood of zero.
\end{claim}
\begin{proof}
Take $q'(f_0)$ to be a pre-image of $q(f_0)$. If $q'(f_0)$ does not meet any critical point before landing on $q(f_0)$, obviously it depends analytically on the parameter. Otherwise there exists $j\neq 1$ such that $q'(f_0)$ is sent to $c_j(f_0)$ and $c_j(f_0)$ is sent to $q(f_0)$. Proposition \ref{1dim-famil} implies that for every parameter $\lambda \in \mathbb{D}$, $c_j(f_{\lambda})$ is preperiodic to $q(f_{\lambda})$, and so $q'(f_{\lambda})$ is indeed a preimage of $c_j(f_{\lambda})$. Since the latter depends analyticly on the parameter, its preimage depends analyticly as well.  
\end{proof} 
Now take $q_{1}(f_0),q_{2}(f_0)$ and $q_{3}(f_0)$ to be three distinct preimages of $q(f_0)$. There exists an analytic family of M\"obius maps $\Gamma_{\lambda}$ sending back the continuation of these three preimages to themselves: $$\Gamma_{\lambda}(q_{m}(f_{\lambda)}))=q_{m}(f_{\lambda})\quad m\in \{1,2,3\}.$$
Since composing with M\"obius maps does not affect normality, the family $\{\lambda\in U \mapsto\Gamma_{\lambda}^{-1}(f_{\lambda}^n(c_{1}(f_{\lambda})))\}_{n\in\mathbb{N}}$ is not a normal family as well, so by Montel's theorem, it cannot avoid all of the three points $q_{1}(f_0),q_{2}(f_0)$ and $q_{3}(f_0)$. Hence, there is a parameter $\lambda^*$, a natural number $l \in \mathbb{N}$ and $ m\in \{1,2,3\}$ such that the following equality holds:
\begin{equation}\label{montel}
\Gamma_{\lambda^*}^{-1}(f_{\lambda^*}^l(c_{1}(f_{\lambda^*})))=q_{m}(f_0).
\end{equation}
 So $f_{{\lambda^*}}^l(c_{1}({\lambda^*}))=q_{m}(f_{\lambda^*})$ which means the critical point $c_{1}(f_{\lambda^*})$ is preperiodic to $q(f_{\lambda^*})$. 

To prove that the parameter ${\lambda^*}$ is in the bifurcation locus, note that the equation (\ref{montel}) cannot holds true in a neighbourhood of ${\lambda^*}$, since otherwise it holds true for any parameter in $U$ but we have assumed that $c_{1}(f_{0})$ is not preperiodic to $q(f_0)$.
\end{proof}
Now let the parameter $\lambda^*$ is chosen so that $c_1(f_{\lambda^*})$ is preperiodic to the periodic point $q(f_{\lambda^*})$ which is the continuation of the periodic point $q$ in the statement of the main lemma. 
\begin{lemma}\label{lemmaparabolicreturn}
Arbitrary close to the parameter $\lambda^*$, there is a parameter $\hat{\lambda}$ such that $f_{\hat{\lambda}}$ has a parabolic periodic point $\hat{q}(f_{\hat{\lambda}})$ and the invariant probability measure supported on the orbit of $\hat{q}(f_{\hat{\lambda}})$ is arbitrary close to the invariant probability measure supported on the orbit of $q(f_{\lambda^*})$.
\end{lemma} 

\begin{proof}
For simplicity, after reparametrizing the family, we may assume that $\lambda^*$ is equal to zero. Without loss of generality, we may also assume that the period of $q(f_{\lambda})$ is equal to one and so it is a fixed point. Otherwise we can repeat the following arguments for a family formed by an iteration of $f_{\lambda}$. Conjugating the family by M\"obius maps, we can assume that $q$ remains a fixed point for all maps in this family. Up to a holomorphic change of local coordinates we can also assume that $f_\lambda$ is linear in a neighbourhood of $q$ and has the following form:
\begin{align}\label{linearization}
f_\lambda(q+z)=\gamma_\lambda z+q,
\end{align} 
where $\gamma_\lambda$ is the multiplier of the repelling fixed point $q$ for the map $f_\lambda$. 
\begin{figure}\label{Fig-Lambda-hat}
    \centering
    \includegraphics[scale=0.5]{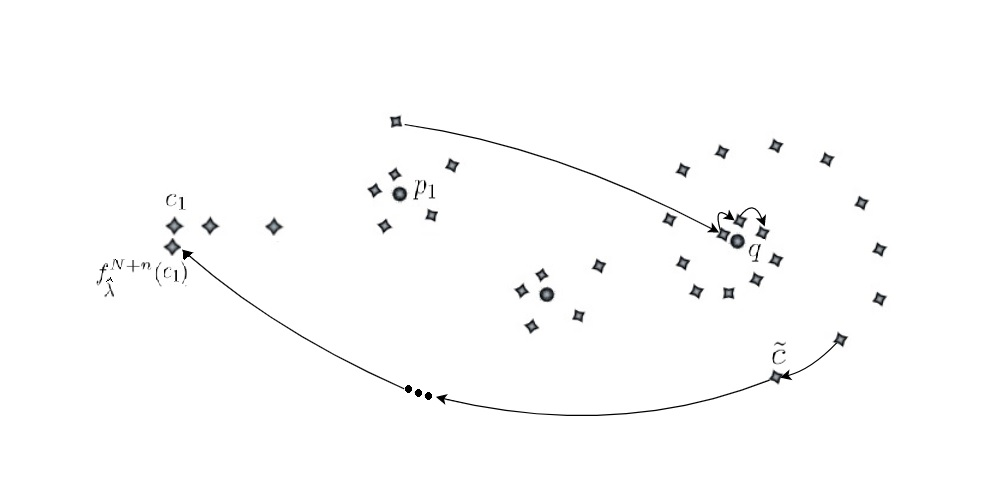}
    \caption{For the map $f_{\hat{\lambda}}$ which is a suitable perturbation of $f_{\lambda^*}$, the orbit of the critical point $c_1$, after staying a long time close to $q$, returns to a small neighbourhood of itself and a parabolic periodic point appears which shadows the orbit of the critical point (see lemma \ref{lemmaparabolicreturn}).}
    \label{Lambda-hat}
\end{figure} 
Next note that since for the map $f_0$ the pre-images of any point accumulates to any point in the Julia set, and the Julia set is the whole Riemann sphere, arbitrary close to $q$, there are pre-images of the critical point $c_1(f)$. We choose one of this pre-images $\tilde{c}$, which is in the linearization domain of $q$. We can also assume the change of coordinates around $q$ is so that the point $\tilde{c}$ stays a preimage of $c_1(f_\lambda)$ for $\lambda$ close to zero. 

Since $c_1(f_0)$ is preperiodic to $q$, there is a natural number $N\in \mathbb N$ such that $f_0^N(\tilde{c})=q$, and also since $\tilde{c}$ meets only one critical point $c_1(f_0)$ (which is simple) before landing on $q$, the Taylor expansion of $f^N_\lambda(z)$ around $z=\tilde{c}$ and $\lambda=0$ has the following form:
 \begin{align}\label{taylor exp}
 f^N_\lambda(\tilde{c}+z)= q+A_\lambda+B_\lambda z^2+z^3\epsilon_\lambda(z),
 \end{align}
where $A_\lambda$, $B_\lambda$ and $\epsilon_\lambda(z)$ depend holomorphically on $\lambda$ and $z$. $A_\lambda$ is zero at $\lambda=0$ but it is not identically zero in a neighbourhood of $\lambda=0$. This is true because $c_1(f_\lambda)$ is not persistently prepriodic to $q$, and so $A_\lambda=\lambda^j\hat{A}(\lambda),$ for some holomorphic map $\hat{A}$ with $\hat{A}(0)\neq 0$ and for some natural number $j\in \mathbb{N}$. 
 On the other hand since $\tilde{c}$ meets only one critical point, which is simple, before landing on $q$, there is no first order term in equation \ref{taylor exp} and also $B_0\neq 0$.
 
 By equation \ref{linearization} 
\begin{align}
f^{N+n}_\lambda(\tilde{c}+z)=q+\gamma_\lambda^n A_\lambda + \gamma_\lambda^n B_\lambda z^2+ \gamma_\lambda^n z^3 \epsilon_\lambda(z).
\end{align} 
Now for each $n\gg 1$, we are going to find a parameter $\lambda_n$ close to zero such that the map $f_{\lambda_n}$ has a parabolic periodic point close to $\tilde{c}$ with period $n+N$ and multiplier equal to one. We find this parameter so that the parabolic periodic point spends most of its time close to the fixed point $q$. For this purpose we need to solve the following system of equations:
\begin{align}
& f_{\lambda}^{N+n}(\tilde{c}+z)=\tilde{c}+z, \label{fixed.pt.equation} \\
 & \large{(}f_\lambda^{n+N}\large{)}'(\tilde{c}+z)=2\gamma_\lambda^n B_\lambda z+ 3\gamma_\lambda^n z^2 \epsilon_\lambda(z)+\gamma_\lambda^n z^3 \large{(}\epsilon_\lambda\large{)}'(z)=1.\label{multipelier.equ}
\end{align} 
From the second equation we obtain 
\begin{align}\label{unif.contracting}
z=\frac{1}{2\gamma_\lambda^n B_\lambda}-\frac{3z^2\epsilon_\lambda(z)-z^3\epsilon_\lambda '(z)}{2B_\lambda}:=G_{n,\lambda}(z).
\end{align}
Using this equation we can find $z$ implicitly in terms of $\lambda$ . Fix a sufficiently small neighbourhood $\mathcal{U}$ of $\lambda=0$ and a small neighbourhood $W$ of $z=0$ such that for large $n$ and for any $\lambda\in \mathcal{U}$ the map $G_{n,\lambda}$ is uniformly contracting on $W$. So for each $n$ and $\lambda$ the map $G_{n,\lambda} $ has a unique fixed point $z_n(\lambda)$. To estimate the norm of this fixed point, using the equation \ref{unif.contracting} we obtain 
$$z_n(\lambda)(1+\frac{3z_n(\lambda)\epsilon_\lambda(z_n(\lambda))-z_n^2(\lambda) \epsilon_\lambda '(z_n(\lambda))}{2B_\lambda})=\frac{1}{2\gamma_\lambda^n B_\lambda},$$
so the norm of $z_n(\lambda)$ is of $O(\frac{1}{|\gamma_\lambda^n|})$. 
Now to find $\lambda_n$ we insert $z_n(\lambda)$ into equation \ref{fixed.pt.equation}: 
$$\frac{z_n(\lambda)+\tilde{c}-q}{\gamma_\lambda ^n}-B_\lambda z_n^2(\lambda)- z_n^3(\lambda) \epsilon_\lambda(z_n(\lambda))=A_\lambda=\hat{A}_\lambda \lambda ^ j.$$
So 
\begin{align}\label{lambda_n.equation}
\lambda ^ j= \frac{1}{\hat{A}_\lambda}\bigg(\frac{\tilde{c}-q}{\gamma_\lambda^n }+\frac{z_n(\lambda)}{\gamma_\lambda^n}-B_\lambda z_n^2(\lambda)-z_n^3(\lambda)\epsilon_\lambda(z_n(\lambda))\bigg):=H_n(\lambda).
\end{align}
since the sequence of maps $\lambda^j-H_n(\lambda)$ converges uniformly on $\mathcal{U}$ to the map $\lambda ^ j$, by Hurwitz theorem we conclude that for $n$ large enough, the equation $\lambda^j-H_n(\lambda)=0$, has $j$ solutions counted with multiplicity. Let $\lambda_n$ be one of these solutions. The pair $(\lambda_n,z(\lambda_n))$ solves both equations \ref{fixed.pt.equation} and \ref{multipelier.equ} so $z_n(\lambda_n)$ is a parabolic periodic point of $f_{\lambda_n}$ with period $n+N$. It remains to show that this periodic point spends most of its time close to the fixed point $q$.

Considering the fact that the norm of $z_n(\lambda)$ is of $O(\frac{1}{|\gamma_\lambda^n|})$ the equation \ref{lambda_n.equation} implies that the norm of $\lambda _n ^ j $ and hence the norm of $A_{\lambda_n}$ are of $O(\frac{1}{|\gamma_\lambda^n|})$ and so the distance between $f^N _{\lambda_n}(\tilde{c}+z_n(\lambda_n))$ and the fixed point $q$ is of this order as well. This shows that the orbit of $z_n(\lambda_n)$ stays $n-O(1)$ iterations close to $q$. Note that since $N$ is fixed, by increasing $n$ the proportion of times that this parabolic periodic point spends close to $q$ tends to 1 and so we are done. 
\end{proof}

The following lemma describes the statistical behavior of Lebesgue a.e. point for the dynamics $f_{\hat \lambda}$, where the parameter $\hat{\lambda}$ is given by Lemma \ref{lemmaparabolicreturn}.
\begin{lemma}\label{Urbanski}
Under the iteration of the map $f_{\hat{\lambda}}$ the empirical measures of Lebesgue almost every point converges to the invariant probability measure supported on the orbit of the parabolic periodic point $\hat{q}(f_{\hat{\lambda}})$ .
\end{lemma}
\begin{proof}
Let $\tilde{U}$ be an immediate basin of attraction of the parabolic periodic point $\hat{q}(f_{\hat{\lambda}})$. By Theorem 10.15 in \cite{Mi06}, the domain $\tilde{U}$ contains a critical point of the map $f_{\hat{\lambda}}$. The only critical point which can live in $\tilde{U}$ is $c_{1}(\hat \lambda)$, because the other ones are preperiodic to repelling periodic points and so are in the Julia set. Assume for the sake of contradiction that there exists a Fatou component $\tilde{V}$ of $f_{\hat{\lambda}}$ which has an orbit disjoint from $\tilde{U}$. By Sullivan's classification of Fatou components for rational maps, the domain $\tilde{V}$ should be a preimage of a periodic Fatou component $\tilde{W}$. The component $\tilde{W}$ cannot be neither a component of the immediate attracting basin of an attracting periodic point nor a component of an immediate attracting basin of a parabolic periodic point, because otherwise it should contain a critical point other than $c_1(\hat \lambda)$ in its forward orbit which is not possible. Since the boundary of a Siegel disk or a Herman ring is accumulated by the orbit of a critical point, the component $\tilde{W}$ cannot be neither of these cases as well. But these are the only possible cases, which is a contradiction. 

Consequently, the set $\bigcup_{n\geqslant 0}f_{\hat{\lambda}}^{-n}(\tilde{U})$ is the whole Fatou set. 
Next note that every critical point of $f_{\hat{\lambda}}$ is non-recurrent. In \cite{przytycki2001porosity} it is proved that a rational map with no recurrent critical point has a Julia set with Hausdorff dimension less than two or a Julia set equal to $\hat{\mathbb{C}}$. 
As the Fatou set of $f_{\hat \lambda}$ is non empty, the Julia set of $f_{\hat \lambda}$ has Hausdorff dimension less than two and in particular has zero Lebesgue measure. This means that almost every point $x\in \hat{\mathbb{C}}$ eventually fall into $\tilde{U}$ , and will be attracted by the orbit of $\hat{q}(f_{\hat{\lambda}})$.
\end{proof}
\begin{remark}
The map $f_{\hat{\lambda}}$ is in the set $\overline{\kappa}_d$. 
\end{remark}
\begin{proof}
 Since the parabolic periodic point for $f_{\hat{\lambda}}$ is not persistent, the parameter $\hat{\lambda}$ is in the bifurcation locus $B$. So again using a normal family argument as in Lemma \ref{argument_normality}, it can be shown that $f_{\hat{\lambda}}$ is approximated by maps $f_{\lambda}$, for which the critical point $c_1(f_{\lambda})$ is preperiodic to a repelling periodic point. This means that $f_{\lambda}\in \kappa_d$ and hence $f_{\hat{\lambda}}\in \overline{\kappa}_d$. 
\end{proof}

To finish the proof of the main lemma, note that by Lemma \ref{Urbanski} and Lemma \ref{lemmaparabolicreturn}, the limit of the empirical measures of almost every point for the map $f_{\hat{\lambda}}$ is close to $e^f_{\infty}(q)$. And moreover, by the previous remark, $f_{\hat{\lambda}}$ is in $\overline{\kappa}_d$. So the map $f$ statistically bifurcates toward $\delta_{e^f_{\infty}(q)}$ with perturbations in $\overline{\kappa}_d$. 
\end{proof}
\section{Periodic measures are dense in $\mathcal M_1(f)$}\label{sec.density-per-meas}
The aim of this section is to prove Proposition \ref{per-meas-density}. Through out this section we assume that $f$ is a strictly postcritically finite rational map of degree $d\geq 2$. Since $f$ has no periodic critical point, it has at least one critical point $c\in \mathcal{C}(f)$ which is not in the post critical set $\mathcal{P}(f)$. So the set $f^{-1}(\{c\})$ has $d$ elements, and since $d\geq2$, the set $\mathcal A:=\mathcal{P}(f)\cup \mathcal{C}(f)\cup f^{-1}(\{c\})$ has at least three elements. The Riemann surface $\hat{\mathbb{C}}\setminus \mathcal A $ is hence a hyperbolic Riemann surface and has the Poincar\'e disk $\mathbb{D}$ as a universal cover. Let us fix a covering map $\pi:\mathbb{D}\to \hat{\mathbb{C}}\setminus \mathcal A $. 

For any point $x\in \hat{\mathbb{C}}\setminus \mathcal A $ and any of its $d$ preimages $y$, the map $f$ is a local diffeomorphism from a neighborhood of $y$ onto a neighborhood of $x$. Thus its inverse branch is well defined and 
 can be locally lifted to the universal covering. We claim that this map can be extended to a map $F:\mathbb{D}\to\mathbb{D}$ satisfying the following property:
\begin{align}\label{lift-of-inv-branch}
 f\circ \pi \circ F=\pi. 
\end{align}

To see this, choose $\tilde{x}\in \pi^{-1}(\{x\})$ and $\tilde{y}\in \pi^{-1}(\{y\})$, and define $F(\tilde{x})=\tilde{y}$. To define $F$ on an arbitrary point $\tilde{z}\in \mathbb{D}$, consider a curve $\gamma:[0,1]\to\mathbb{D}$ with $\gamma(0)=\tilde{x}$ and $\gamma(1)=\tilde{z}$. Then by projecting this curve to $\hat{\mathbb{C}}\setminus \mathcal A $ and using the continuation of the inverse branch sending $x$ to $y$, we obtain a curve in $\hat{\mathbb{C}}\setminus \mathcal A $ starting at $y$ and ending at a point in $f^{-1}(\{\pi(\tilde{z})\})$. This new curve has a lift to the universal cover, which starts at $\tilde{y}$. We define $F(\tilde{z})$ as the endpoint of the latter curve. The map $F$ is well defined since for any other curve ${\gamma}'$ joining $\tilde{x}$ to $\tilde{z}$, the loop ${(\gamma}')^{-1}\circ \gamma$ is contractible in $\mathbb{D}$, so its projection $\pi({(\gamma}')^{-1}\circ \gamma)$, is a contractible loop in $\hat{\mathbb{C}}\setminus \mathcal A$ as well. The inverse image of this loop under the continuation of the branch of $f^{-1}$ sending $x$ to $y$ is then contractible in $\hat{\mathbb{C}}\setminus \mathcal A$, and so lifts to a closed loop in $\mathbb{D}$, starting from $\tilde{y}$. This Shows that we obtain the same points for $F(\tilde{z})$ using both $\gamma$ and $\gamma'$, and hence $F$ is well defined. By definition, it is obvious that equation (\ref{lift-of-inv-branch}) holds for $F$.

We denote the hyperbolic metric on the Poincar\'e disk $\mathbb{D}$ by $\tilde{d}_h$.
Recall that any Deck transformation of the covering $\pi : \mathbb D \to \hat{\mathbb{C}}\setminus \mathcal A$ is a biholomorphism, and so it leaves invariant the Poincar\'e metric $\tilde{d}_h$. Thus we can push forward the metric $\tilde{d}_h$ and obtain a metric $d_h$ on $\Hat{\mathbb{C}}\setminus \mathcal A$.
\begin{lemma}
For the metric $\tilde{d}_h$, the derivative $DF(z)$ is contracting at every $z\in \mathbb{D}$.
\end{lemma}
\begin{proof}
Schwarz lemma implies that if $F$ is not an isomorphism of the Poincar\'e disk, then $DF(z)$ is $\tilde{d}_h$-contracting for every $z\in \mathbb{D}$. We are going to show that $f$ is not surjective and hence can not be an isomorphism. Choose a point $x\in \mathcal A$ which is a preimage of the critical point $c$. Let $y$ be a preimage of $x$. We recall that $c$ is not in the postcritical set, so $y$ cannot be in $\mathcal A$. Now take any point $\tilde{y}\in {\pi}^{-1}({y})$. Since we have 
$$f\circ \pi \circ F(\mathbb{D})=\pi(\mathbb{D})=\hat{\mathbb{C}}\setminus \mathcal A,$$
$\tilde{y}$ cannot be in the range of $F$.
\end{proof}
The following corollaries are immediate consequences of the previous lemma:
\begin{corollary}\label{contracting-inv-brnch}
At every point $x\in \hat{\mathbb{C}}\setminus \mathcal{A}$, any inverse branch of $f$ has a contracting derivative for the metric $d_h$.
\end{corollary}
\begin{corollary}\label{all-per-repell}
Any periodic point of $f$ is repelling.
\end{corollary}
\begin{proof}[Proof of Proposition \ref{per-meas-density}]
We shall prove that every probability measure of $f$ can be approximated by invariant probability measures supported on the orbit of a periodic point. First let us show this for 
the case where the probability measure is ergodic. 

\begin{lemma}\label{ergodic-approx}
Any ergodic invariant probability measure $\mu\in{\mathcal{M}}_1(f)$, can be approximated by invariant probability measures supported on the orbit of a periodic point. 
\end{lemma}
\begin{proof}
Since $\mu$ is ergodic, we can find a point $x$ in the support of $\mu$ which is \emph{regular} for $\mu$ meaning that the sequence of the empirical measures $\{e_{n}^f(x)\}_{n\in\mathbb{N}}$ converges to $\mu$. If the orbit of $x$ intersects the set $\mathcal A$, the point $x$ is eventually periodic and in fact is a periodic point in $\mathcal A$. In this case, the measure $\mu$ is itself a measure supported on the orbit of the periodic point $x$. So let us assume that the orbit of $x$ is disjoint from $\mathcal A$. For small $r>0$, let $B_r(x)$ be the ball of radius $r$ about $x$ with respect to the metric $d_h$. Since the metric $d_h$ is complete, the closure of $B_r(x)$ is included in $\hat{\mathbb{C}}\setminus \mathcal A$. Note that there are only finite inverse branches of $f$, and we can use Corollary \ref{contracting-inv-brnch} to conclude that there is a number $0<\alpha<1$ such that any inverse branch of $f$ over $B_r(x)$ is at least $\alpha$-contracting.

On the other hand, since $x$ is in the support of $\mu$, and also a regular point for this measure, its orbit returns infinitely many times to its hyperbolic $\frac{r}{4}$-neighbourhood. Let $m\in\mathbb{N}$ be such that $\alpha^m<\frac{1}{2}$. Choose $n\in\mathbb{N}$ such that the orbit of $x$ up to $n$ iterations contains at least $m+1$ points inside $B_{\frac{r}{4}}(x)$, including $f^n(x)$. Let $U_0:= B_{r/2}(f^n(x))$, and for each $1\leq i \leq n$, denote the connected component of $f^{-i}(U_0)$ containing $f^{n-i}(x)$ by $U_i$. Since any inverse branch of $f$ is non-expanding, any $U_i$ is contained in a ball of radius $\frac{r}{2}$ around $f^{n-i}(x)$. And so when $f^{n-i}(x)$ is $\frac{r}{4}$ close to $x$, $U_i$ is contained in $B_r(x)$. This implies that $f^{-1}$ sending $U_i$ to $U_{i+1}$ is $\alpha$-contracting and so the branch $g$ of $f^{-n}$ from $U_0$ to $U_n$ is $\alpha^m$-contracting. Recalling that $\alpha^m<\frac{1}{2}$, this implies that $U_n$ is in $\frac{r}{4}$-neighbourhood of $x$. But $U_0$ covers the $\frac{r}{4}$-neighbourhood of $x$, so $g$ sends $U_0$ into itself, and is $\alpha^m$-contracting. Thus there is a a fixed point $p$ of $g$ in the closure of $U_0$. 
This fixed point is an $n$-periodic point of $f$ satisfying: 

\begin{align}\label{closing-orbit}
\forall i\in\{0,...,n\},\quad d_h(f^i(x),f^i(p))<\frac{r}{2}.
\end{align}

But there is a constant $C>0$ (depending only on $\mathcal A$) such that for any two points $x$ and $y$ in $\hat{\mathbb C} \setminus \mathcal A$ we have:

$$d(x,y)<Cd_h(x,y),$$
where $d(x,y)$ is the standard spherical metric between $x$ and $y$ in $\hat{\mathbb{C}}$. We refer the reader to \cite{bonk1996bounds}. So the orbit of $x $ and the periodic point $p$ are close to each other in the spherical metric:
$$\forall i\in\{0,...,n\},\quad d(f^i(x),f^i(p))<C\frac{r}{2},$$
and hence 
$$d_w(e^f_n(x),e^f_n(p))<C\frac{r}{2}.$$
By choosing $r$ small enough and $n$ large enough, we can guarantee that $e^f_n(x) $ is close to $\mu$. This shows that $\mu$ can be approximated by the invariant measures supported on the orbit of periodic points. 
\end{proof}
The final step in the proof of Proposition \ref{per-meas-density} is to show that every invariant measure of $f$ can be approximated by the invariant measures supported on the orbit of only one periodic point. For this we show that any finite convex combination of ergodic invariant measures of $f$ can be approximated by such measures, and since, the finite convex combinations of ergodic invariant measures are dense in the set of invariant measures of $f$ (according to ergodic decomposition theorem, any invariant measure cam be written as an integral of erg), Proposition \ref{per-meas-density} follows.

Let $\mu_1,...,\mu_k$ be $k$ ergodic invariant measures, and $\mu=\sum_{i=1}^{k}c_i\mu_i$ a convex combination of these measures for some $0\leq c_i \leq 1$ with $\sum_{i=1}^k c_i=1$. By lemma \ref{ergodic-approx} for each $1\leq i \leq k$, there exists a periodic point $p_i$ such that $d_w(\mu_i,e^f_{\infty}(p_i))$ is arbitrary small and hence $d_w(\mu,\sum_{i=1}^k c_i e^f_{\infty}(p_i))$ is small. So for our purpose, it is enough to show that the measure $\sum_{i=1}^k c_i e^f_{\infty}(p_i)$ can be approximated by invariant probability measures which are supported on the orbit of only one periodic point. To show this, For technical reasons it is better to bring into play another repelling periodic point $p_0$, which is not in the post critical set $\mathcal{P}(f)$. 

Since the Julia set of $f$ is the whole Riemann sphere, the set of all preimages of each periodic point $p_i$ is dense in $\hat{\mathbb{C}}$, and in particular has a point in the linearization domain of the other $k$ periodic points. So we can find $\epsilon>0$ such that the preimages of $\epsilon$-neighbourhood of $p_i$ has a connected component in the linearization domain of $p_{i+1}$ (for i=k, consider $p_0$ instead of $p_{i+1}$). Let us denote the $\epsilon$-neighbourhood of $p_i$ by $U_i$. Now note that preimages of $U_i$ has indeed a connected component in $U_{i+1}$ because any subset of the linearization domain, has preimages converging to the periodic point $p_i$. Take $l_i\in \mathbb{N}$ such that $f^{-l_i}(U_i)$ has a connected component in $U_{i+1}$ (in $U_0$, for $i=k$).

Now we find a periodic point, in a backward orbit of $U_0$ which returns to itself. For each set of natural numbers $\{n_1,...,n_k\}\subset\mathbb{N}$ such that $n_i$ is divisible by the period of $p_i$, consider the following backward orbit of $U_0$: the set $U_0$ is sent by $f^{-l_0}$ into $U_1$. Then for each $1\leq i \leq k$ spends $n_i$ backward iterations in the linearization domain of $p_i$, and then by $f^{-l_i}$ goes from $U_i$ to $U_{i+1}$ (to $U_0$ for $i=k$). So finally, we will obtain a preimage $\tilde{U}_0$ of $U_0$ in itself. Since $U_0$ does not intersect the post critical set, there is no critical point in the preimages of this set, and the inverse branch sending $U_0$ to $\tilde{U}_0$ is a homeomorphism, and in particular, by Brouwer fixed point theorem, it has a fixed point $p$. This fixed point is a periodic point for the map $f$ with the period equal to $N:=l_0+\sum_{i=1}^k l_i+n_i$. This periodic point spends $n_i$ iteration close to the orbit of $p_i$, so since the sum $\sum_{i=0}^k l_i$ is bounded, by choosing very large integers $n_i$ such that for each $i$, the number $\frac{n_i}{N}$ is close $c_i$, we can guarantee that $e^f_{N}(p)$ is arbitrarily close to $\sum_{i=1}^k c_i e^f_{\infty}(p_i)$. This finishes the proof of Proposition \ref{per-meas-density}. 

\end{proof}


\bibliographystyle{plain}
\bibliography{references}

\end{document}